\documentclass[12pt]{amsart}
\usepackage{amsmath, amsthm, amscd, amsfonts}
\usepackage{color}

\newtheorem{theorem}{Theorem}[section]
\newtheorem{lemma}[theorem]{Lemma}
\newtheorem{proposition}[theorem]{Proposition}

\theoremstyle{definition}
\newtheorem{definition}[theorem]{Definition}

\newtheorem{construction}[theorem]{Construction}
\theoremstyle{remark}

\newtheorem{remark}[theorem]{Remark}
\numberwithin{equation}{section}

\newfont{\kh}{msbm10}

\begin{document}
\title[Completely positive maps on Hilbert modules]
{Completely positive maps on Hilbert modules over pro-C*-algebras}
\author{Kh. Karimi}
\address{Khadijeh Karimi,
\newline Department of Mathematics, Shahrood University, P.
O. Box 3619995161-316, Shahrood, Iran.}
\email{kh\underline{\space}karimi5005@yahoo.com}
\author{K. Sharifi}
\address{Kamran Sharifi,
\newline Department of Mathematics,
Shahrood University, P.
O. Box 3619995161-316, Shahrood, Iran.
\newline Mathematisches Institut,
Fachbereich Mathematik und Informatik der Universit\"{a}t M\"{u}nster,
Einsteinstrasse 62, 48149 M\"{u}nster, Germany.}
\email{sharifi.kamran@gmail.com}


\begin{abstract}
We derive Paschke's GNS construction for completely positive maps
on unital pro-C*-algebras from the KSGNS construction, presented by M. Joita
[J. London Math. Soc.  {\bf 66} (2002), 421--432],
and then we deduce an analogue of Stinespring theorem for
Hilbert modules over pro-C*-algebras. Also, we obtain a Radon-Nikodym type
theorem for operator valued completely positive maps on Hilbert
modules over pro-C*-algebras.
\end{abstract}
\maketitle


\section{Introduction}
Completely positive maps are the natural generalization of positive linear functionals.
These maps are extremely applied in the modern theory of C*-algebras and
mathematical model of quantum probability. A completely positive map $\varphi:A\to B$ of
C*-algebras is a linear map with the property that $[\varphi(a_{ij})]_{i,j=1}^{n}$
is a positive element in the C*-algebra $M_{n}(B)$ of all $n\times n$ matrices with
entries in $B$ for all positive matrices $[a_{ij}]_{i,j=1}^{n}$ in $M_{n}(A)$, $n\in\Bbb{N}$.
Given a C*-algebra $A$, the Gelfand-Naimark-Segal construction
(or GNS-construction) establishes a correspondence between cyclic representations of $A$ on
Hilbert spaces and positive linear functionals on $A$. This fundamental theorem has been generalized
for a completely positive linear map from $A$ into $B(H)$ (respectively, from $A$ into a C*-algebra $B$)
to get a representation of $A$ on a Hilbert space
(respectively, on a Hilbert $B$-module) by Stinespring \cite{Stinespring}
(respectively, Paschke \cite{Paschke}). Stinespring 
showed that an operator valued completely positive map $\varphi$ on a unital
C*-algebra $A$ is of the form $V_{\varphi}^{*}\pi_{\varphi}( \cdot ) V_{\varphi}$, where $\pi_{\varphi}$
is a representation of $A$ on a Hilbert space $H_{\varphi}$ and $V_{\varphi}$
is a bounded linear operator. A
version of Stinespring theorem for a class of maps on Hilbert modules over
unital C*-algebras, which are known operator-valued completely positive maps
on Hilbert C*-modules, has been considered by \cite{Asadi, Bhat}.
Skeide \cite{Skeide} has been obtained a very quick proof of the
result of \cite{Bhat} by using induced representations of Hilbert C*-modules.

The theory of completely positive maps on pro-C*-algebras has been studied
systematically in the book \cite{Joita1} and the paper \cite{Joita2} by Joita.
Maliev and Pliev \cite{Maliev} obtained a Stinespring theorem for Hilbert modules over
pro-C*-algebras by extending the methods of \cite{Bhat} from the case of C*-algebras to the
case of pro-C*-algebras. We generalize the Paschke's GNS-construction
to completely positive maps on pro-C*-algebras which enables us to establish
another proof for the Stinespring theorem for
Hilbert modules over pro-C*-algebras.

Let us quickly recall the definition of pro-C*-algebras and Hilbert modules over them.
A pro-C*-algebra is a complete Hausdorff complex
topological $*$-algebra $A$ whose topology is determined by its continuous C*-seminorms in the
sense that the net $\{a_{i}\}_{i \in I}$ converges to $0$ if and only
if the net $\{p(a_i)\}_{i \in I}$ converges to $0$ for every
continuous C*-seminorm $p$ on $A$. For example
the algebra $C(X)$ of all continuous
complex valued functions on a compactly generated space (or a CW complex $X$) with
the topology of compact convergence and the cartesian
product $\prod _{\alpha\in I}A_{\alpha}$ of C*-algebras $A_{\alpha}$ with the product topology
are pro-C*-algebras \cite[\S 7.6]{fra}.
Pro-C*-algebras appear in the study
of certain aspects of C*-algebras such as tangent algebras of C*-algebras, domain of
closed $*$-derivations on C*-algebras, multipliers of
Pedersen's ideal, and noncommutative analogous of classical Lie groups.
These algebras were first introduced by Inoue \cite{ino} who called them
locally C*-algebras and studied more in \cite{Apostol, fra, phil1}
with different names.
A (right) {\it pre-Hilbert module} over a pro-C*-algebra
$A$ is a right $A$-module $E$,
compatible with the complex algebra structure, equipped with an
$A$-valued inner product $\langle \cdot , \cdot \rangle
: E \times E \to A \,, \ (x,y) \mapsto \langle x,y
\rangle$, which is $A$-linear in the second variable $y$
and has the properties:
$$ \langle x,y \rangle=\langle y,x \rangle ^{*}, \ {\rm and} \
 \langle x,x \rangle \geq 0 \ \ {\rm with} \
   {\rm equality} \ {\rm if} \ {\rm and} \ {\rm only} \
   {\rm if} \ x=0.$$
A pre-Hilbert $A$-module $E$ is a Hilbert
$A$-module if $E$ is complete with respect to the
topology determined by the family of seminorms $ \{
\overline{p}_E \}_{p \in S(A)}$  where $\overline{p}_E
( \xi) = \sqrt{ p( \langle \xi, \xi \rangle)}$, $ \xi \in E$.
Hilbert modules over pro-C*-algebras have
been studied in the book \cite{mar1} and the
papers \cite{phil1, SHAMal}.

There is a natural partial ordering on the set of all operator valued
completely positive maps on C*-algebras, defined by $\psi\leq \varphi$
if $\varphi-\psi$ is completely positive. Arveson \cite{Arveson}
characterized this relation in terms of the Stinespring construction
associated to each completely positive map and introduced a notion of
Radon-Nikodym derivative for operator valued completely positive maps
on C*-algebras. Indeed, he showed
that in the unital case, $\psi\leq\varphi$ if and only if there is a unique
positive contraction $\Delta_{\varphi}(\psi)$ (known as Radon-Nikodym
derivative of $\psi$ with respect to $\varphi$) in the commutant of $\pi_{\varphi}(A)$ such that $\psi( \cdot)=V_{\varphi}^{*}\Delta_{\varphi}(\psi)\pi_{\varphi}(\cdot)V_{\varphi}$\,, cf. \cite[Theorem 1.4.2]{Arveson}.
Joita \cite{Joita5} defined a preorder relation in
the set of all operator valued completely positive maps on
Hilbert C*-modules and extended the Radon-Nikodym type theorem for these maps.

In this paper we first present some definitions and basic facts about pro-C*-algebras
and Hilbert modules over them. In Section 3, by using the concept of induced representations
of Hilbert C*-modules, we deduce Stinespring representation theorem
of pro-C*-algebras from Paschke's GNS-construction. Then we
obtain a version of Stinespring representation theorem for
Hilbert modules over pro-C*-algebras. Finally in section 4,
we generalize the Radon-Nikodym theorem for operator valued
completely positive maps on Hilbert modules over pro-C*-algebras.


\section{Preliminaries}
Let $A$ be a pro-C*-algebra, $S(A)$ be the set of all
continuous C*-seminorms on $A$ and $p\in S(A)$. We set
$N_{p}=\{a\in A:\ p(a)=0\}$ then $A_p=A/N_{p}$ is a C*-algebra in the norm induced
by $p$. For $p,q\in S(A)$ with $p\geq q$, the surjective
morphisms $\pi_{pq}: A_p \to A_q$ defined by $\pi_{pq}(a+N_{p})=a+N_{q}$
induce the inverse system $\{A_p; \pi_{pq}\}_{p,q\in S(A), \, p\geq q}$
of C*-algebras and $A = \varprojlim_{p}A_p$, i.e.
the pro-C*-algebra $A$ can be identified with $\varprojlim_{p}A_p$.
The canonical map from $A$ onto $A_p$ is
denoted by $\pi_p$ and $a_{p}$ is reserved to denote $a+N_{p}$.
A morphism of pro-C*-algebras is a continuous morphism of $*$-algebras.
An isomorphism of pro-C*-algebras is a morphism of pro-C*-algebras
which possesses an inverse morphism of pro-C*-algebras.

We denote by $M_{n}(A)$ the set of all $n\times n$ matrices over $A$. The set $M_{n}(A)$
with the usual algebraic operations and the topology obtained by regarding it
as a direct sum of $n^{2}$ copies of $A$ is a pro-C*-algebra. Moreover, it
can be identified with $\varprojlim_{p}M_{n}(A_p)$. Thus the topology on $M_{n}(A)$ is
determined by the family of C*-seminorms $\{p^{(n)}\}_{p\in S(A)}$,
where $p^{(n)}([a_{ij}])=\|[\pi_{p}(a_{ij})]\|_{M_{n}(A_p)}$, $[a_{ij}]\in M_{n}(A)$.

A representation of a pro-C*-algebra $A$ is a continuous $*$-morphism $\varphi: A \to B(H)$,
where $B(H)$ is the C*-algebra of all bounded linear maps on a Hilbert space $H$.
If $(\varphi,H)$ is a representation of $A$, then there is $p\in S(A)$ such that
$\|\varphi(a)\|\leq p(a)$, for all $a\in A$.
The representation $(\varphi_{p},H)$ of $A_{p}$, where $\varphi_{p} \circ  \pi_{p}=\varphi$
is called a representation of $A_{p}$ associated to $(\varphi,H)$.
We refer to \cite{fra, mar4} for more detailed information about the
representation of pro-C*-algebras.

Suppose $E$ is a Hilbert $A$-module and $\langle E,E\rangle$ is the
closure of linear span of $\{ \langle x,y\rangle: ~ x,y\in E \}$.
The Hilbert $A$-module $E$ is called {\it full} if $\langle E,E\rangle=A$.
One can always consider any Hilbert $A$-module as a full Hilbert module
over pro-C*-algebra $\langle E,E\rangle$.
For each $p\in S(A), N_{p}^{E}=\{\xi\in E: \ \bar{p}_{E}(\xi)=0\}$
is a closed submodule of $E$ and $E_{p}=E/N_{p}^{E}$ is a
Hilbert $A_{p}$-module with the action $(\xi+N_{p}^{E})\pi_{p}(a)=\xi a+N_{p}^{E}$
and the inner product $\langle\xi+N_{p}^{E},\eta+N_{p}^{E}\rangle=\pi_{p}(\langle\xi,\eta\rangle).$
The canonical map from $E$ onto $E_{p}$ is denoted by $\sigma_{p}^{E}$ and
$\xi_{p}$ is reserved to denote $\sigma_{p}^{E}(\xi).$
For $p,q\in S(A)$ with $p\geq q$, the surjective morphisms
$\sigma_{pq}^{E}:E_{p} \to E_q$ defined by
$\sigma_{pq}^{E}(\sigma_{p}^{E}(\xi))=\sigma_{q}^{E}(\xi)$
induce the inverse system $\{E_{p};\  A_p;\  \sigma_{pq}^{E},\ \pi_{pq} \}_{p,q\in S(A), \, p\geq q}$\,
of Hilbert C*-modules.
In this case, $\varprojlim_{p}E_{p}$ is a
Hilbert $A$-module which can be identified with $E$.
Let $E$ and $F$ be Hilbert $A$-modules and $T:E\to F$ be an $A$-module map.
The module map $T$ is called bounded if for each $p\in S(A)$,
there is $k_p>0$ such that $\bar{p}_{F}(Tx)\leq k_{p}\ \bar{p}_{E}(x)$ for all $x\in E$.
The set $L_{A}(E,F)$ of all bounded adjointable $A$-module maps from $E$
into $F$ becomes a locally convex space with the topology defined by the
family of seminorms $\{\tilde{p}\}_{p\in S(A)}$, where $\tilde{p}(T)=\|(\pi_{p})_{*}(T)\|_{ L_{A_p}(E_p,F_p)}$
and $(\pi_{p})_{*}:L_{A}(E,F)\to L_{A_p}(E_p,F_p)$ is defined by
$(\pi_p)_{*}(T)(\xi+N_{p}^{E})=T\xi+N_{p}^{F}$, for
all $T\in L_{A}(E,F)$ and $ \xi\in E$. For $p,q\in S(A)$ with $p\geq q$,
the morphisms $(\pi_{pq})_{*}:L_{A_p}(E_p,F_p)\to L_{A_q}(E_q,F_q)$
defined by $(\pi_{pq})_{*}(T_{p})(\sigma_{q}^{E}(\xi))=\sigma_{pq}^{F}(T_p(\sigma_{p}^{E}(\xi)))$ induce
the inverse system $\{ L_{A_p}(E_p,F_p);\ (\pi_{pq})_{*}\}_{p,q\in S(A), \, p\geq q}$
of Banach spaces such that $\varprojlim_{p} L_{A_p}(E_p,F_p)$ can be identified to $L_{A}(E,F)$.
In particular, topologizing, $L_{A}(E,E)$ becomes a
pro-C*-algebra which is abbreviated by $L_{A}(E)$.
The set of all compact operators $K_{A}(E)$ on $E$ is defined as the closed
linear subspace of $L_A(E)$ spanned by
$\{\theta_{x,y}: \ \theta_{x,y}(\xi)=x\langle y,\xi\rangle ~ {\rm for~ all}~ x,y,\xi \in E\}$.
 This is a pro-C*-subalgebra and a two sided ideal of $L_{A}(E)$, moreover,
$K_{A}(E)$ can be identified to $\varprojlim_{p}\ K_{A_p}(E_p)$.

Let $E$ and $F$ be Hilbert modules over pro-C*-algebras $A$ and $B$, respectively,
and $\Psi:A\to L_{B}(F)$ be a continuous $*$-morphism. We can regard $F$ as a
left $A$-module by $(a,y)\to \Psi(a)y, ~a\in A, ~y\in F$. The right $B$-module
$E\otimes_{A}F$ is a pre-Hilbert module
with the inner product given by
$\langle x\otimes y, z\otimes t\rangle=\langle y,\Psi(\langle x,z\rangle)t\rangle$.
We denote by $E\otimes_{\Psi}F$ the completion of $E\otimes_{A}F$,
cf. \cite{mar1} for more detailed information.

\section{Stinespring representation theorem}
In this section, we first generalize Paschke's GNS-construction \cite[Theorem 5.2]{Paschke}
to the framework of unital pro-C*-algebras. It is a particular case of the KSGNS construction for completely positive maps on unital pro-C*-algebras, \cite[Theorem 4.6]{Joita2}. Then we deduce Stinespring representation
theorem in the context of pro-C*-algebras and a version of Stinespring representation theorem for Hilbert modules over pro-C*-algebras. For this aim we briefly restate the concept of induced representations of Hilbert modules over pro-C*-algebras from our recent paper \cite{SK}.

Given pro-C*-algebras $A$ and $B$, a
linear map $\varphi:A\to B$ is said to be {\it positive} if $\varphi(a^*a)\geq 0$
for all $a\in A$. If $\varphi^{(n)}:M_{n}(A)\to M_{n}(B)$ defined
by $\varphi^{(n)}([a_{ij}]_{i,j=1}^{n})=[\varphi(a_{ij})]_{i,j=1}^{n}$
is positive, then $\varphi$ is said to be $n$-positive. If $\varphi$
is $n$-positive for all natural numbers $n$, then $\varphi$ is called
a {\it completely positive map} \cite{Bhatt}. Let $H$ be a Hilbert space
and $\varphi:A\to B(H)$ be an operator valued completely positive
map then the condition of positivity \cite{Stinespring} can be written in the form
$$\sum_{i,j=1}^{n}\langle \varphi(a_{i}^{*}a_{j})h_{j},h_{i}\rangle\geq 0,
\ {\rm for~all} ~ h_{j}\in H,~ a_{j}\in A,~ j=1,...,n,~ {\rm and ~ all}~ n\in \Bbb{N}.$$

Let $A$ and $B$ be pro-C*-algebras and let $E$ be a Hilbert $B$-module.
A continuous $*$-morphism from $A$ into $L_{A}(E)$ is called a
{\it continuous representation} of $A$ on $E$.

\begin{theorem}\label{Cp1}
Let $A$ and $B$ be two unital pro-C*-algebras and $\varphi:A \to B$ be a continuous completely positive map.
There is a Hilbert $B$-module $X$, a unital continuous representation
$\pi_{\varphi}$ of $A$ on $X$, $\pi_{\varphi}:A\to L_{B}(X)$, and an element $\xi\in X$ such that
\begin{enumerate}
 \item $\varphi(a)=\langle \xi,\pi_{\varphi}(a)\xi \rangle$ for all $a\in A$;
 \item the set $\chi_{\varphi}=span\{\pi_{\varphi}(a)(\xi b):  a \in A, b \in B \}$ is a dense subspace of $X$.
\end{enumerate}
In this case, we say that $\pi_{\varphi}$ is the Paschke's GNS construction associated to completely positive map $\varphi$.
\end{theorem}

The result follows by taking into account $E=B$ and $\xi=V_{\varphi}(1_{B})$ in the proof of \cite[Theorem 4.6]{Joita2},
where $1_{B}$ is the unit of $B$.

Let $E$ and $F$ be Hilbert modules over pro-C*-algebras $A$ and $B$, respectively
and $\varphi:A\to B$ be a morphism of pro-C*-algebras. A map $\Phi: E \to F$ is
said to be a $\varphi$-morphism if
$\langle \Phi(x), \Phi(y) \rangle=\varphi(\langle x,y\rangle)$, for all $x,y\in E$.
A $\varphi$-morphism $\Phi:E\to F$ is said to be a completely
positive map if $\varphi:A\to B$ be a completely positive map.

\begin{lemma}\label{Cp2}
Let $E$ and $F$ be Hilbert modules over unital pro-C*-algebras $A$ and $B$, respectively,
and let $\Phi:E\to F$ be a completely positive map. Suppose that $X$, $\xi$ and $\pi_{\varphi}$
are the same as Theorem \ref{Cp1}. Then there exists
an isometry $v: E\otimes_{\pi_{\varphi}}X\to F$ such that $v(z\otimes \xi)=\Phi(z)$,
for all $z\in E$.
\end{lemma}

\begin{proof}
For $a,c\in A$, $b,d\in B$ and $z,w \in E$ we have
\begin{eqnarray*}
\langle z\otimes (\pi_{\varphi}(a)(\xi b)),w\otimes (\pi_{\varphi}(c)(\xi d)) \rangle &=&
\langle  \pi_{\varphi}(a)(\xi b), \pi_{\varphi}(\langle z,w \rangle)(\pi_{\varphi}(c)(\xi d)) \rangle
\\ &=& b^{*}\langle \xi,\pi_{\varphi}(\langle za,wc\rangle)\xi \rangle  d
\\ &=& b^{*}\varphi(\langle za,wc\rangle)  d
\\ &=& b^{*}\langle \Phi(za),\Phi(wc)\rangle  d
\\ &=& \langle \Phi(za)b,\Phi(wc)d\rangle.
\end{eqnarray*}
Since ${\rm span} \{\pi_{\varphi}(a)(\xi b):  a \in A, b \in B \}$ is a dense
subspace of $X$, the map $z\otimes (\pi_{\varphi}(a)(\xi b)) \mapsto \Phi(za)b$
defines an isometry $v: E\otimes_{\pi_{\varphi}}X\to F$. In particular,
we find $v(z\otimes \xi)=\Phi(z)$ for all $z\in E$, when $a=1_{A}$ and $b=1_{B}$.
\end{proof}

Let $H$ and $K$ be Hilbert spaces. Then the space $B(H,K)$ of all bounded operators from $H$ into $K$
can be considered as a Hilbert $B(H)$-module with the module action
$(T,S)\to TS$, $T\in B(H,K)$ and $S\in B(H)$ and the inner product defined by
$\langle T,S\rangle=T^{*}S$, $T,S\in B(H,K)$. Murphy \cite{mur} showed that any
Hilbert C*-module can be represented as a submodule of the concrete Hilbert module $B(H,K)$
for some Hilbert spaces $H$ and $K$.
This allows us to extend the notion of a representation from the
context of C*-algebras to the context of Hilbert C*-modules.
Let $E$ and $F$ be two Hilbert modules over C*-algebras $A$ and $B$, respectively,
and let $\varphi:A \to B$ be a morphism of C*-algebras. A $\varphi$-morphism $\Phi:E\to B(H,K)$,
where $\varphi:A \to B(H)$ is a representation of $A$ is called a representation of $E$.
When $\Phi$ is a representation of $E$, we assume that an associated
representation of $A$ is denoted by
the same lowercase letter $\varphi$, so we will not explicitly mention $\varphi$.
Let $\Phi:E\to B(H,K)$ be a representation of a Hilbert $A$-module $E$.
We say $\Phi$ is a non-degenerate representation if
$[ \Phi(E)(H)]=K$ and $[ \Phi(E)^{*}(K)]=H$, where $[ \Phi(E)(H)]$
denotes the closure of ${\rm span} \{ \Phi( \xi )(h) ; ~ \xi \in E, ~h \in H \}$.
Two representations $\Phi_{i}:E\to B(H_{i},K_{i})$ of $E$, $i=1,2$
are said to be unitarily equivalent, if there are
unitary operators $U_{1}:H_{1}\to H_{2}$ and $U_{2}:K_{1}\to K_{2}$,
such that $U_{2}\Phi_{1}(x)=\Phi_{2}(x)U_{1}$ for all $x\in E$.
Representations of Hilbert modules have been investigated in ~\cite{Arambasic,ske1}.

Skeide ~\cite{ske1} recovered the result of
Murphy by embedding of every Hilbert $A$-module
$E$ into a matrix C*-algebra as a lower submodule. He proved that every
representation of $A$ induces a representation of $E$ and a representation of $L_{A}(E)$. We describe his
induced representations as follows.
\begin{construction}\label{const1}
Let $A$ is a C*-algebra and $E$ be a Hilbert $A$-module
and $\varphi: A \to B(H)$ be a $*$-representation
of $A$. Define a sesquilinear form $\langle.,.\rangle$ on the vector space $E\otimes H$ by
$\langle x\otimes h,y\otimes k\rangle=\langle h,\varphi(\langle x,y\rangle)k\rangle_{H},$
where $\langle.,.\rangle_{H}$ denotes the inner product on the Hilbert space $H$.
By \cite[Proposition 3.8]{ske1}, the sesquilinear form is positive and so $E\otimes H$
is a semi-Hilbert space. Then $(E\otimes H)/N_{\varphi}$
is a pre-Hilbert space with the inner product defined by
$$\langle x\otimes h+N_{\varphi} \ ,\  y\otimes k+N_{\varphi}\rangle=\langle x\otimes h,y\otimes k\rangle,$$
where $N_{\varphi}$ is the vector subspace of $E\otimes H$
generated by $\{x\otimes h\in E\otimes H: \langle x\otimes h,x\otimes h\rangle=0 \}$.
Let $K$ be the completion of $(E\otimes H)/N_{\varphi}$ with respect to the above inner product.
We identify the elements $x\otimes h$ with the equivalence classes $x\otimes h+N_{\varphi}\in$  $K$.
Suppose $x\in E$ and $L_{x}h=x\otimes h$ then
$\| L_{x}h\|^2= \langle h,\varphi(\langle x,x\rangle)h\rangle\leq \| h \|^2\|x\|^2$, i.e.
$L_{x}\in B(H,K)$. If $L_{x}^{*}$ be the adjoint of $L_{x}$ then it is easy to show
that $L_{x}^{*}(y\otimes h)=\varphi(\langle x,y\rangle)h$ for every $y\in E$ and $h\in H$.
We define $\eta_{\varphi}: E \to B(H,K)$ by $\eta_{\varphi}(x)=L_{x}$. Then
for $x,x'\in E$,  $h,h'\in H$ and $a\in A$ we have
$\langle \eta_{\varphi}(x),\eta_{\varphi}(x')\rangle=\varphi(\langle x,x'\rangle)$ and
$\eta_{\varphi}(xa)=\eta_{\varphi}(x)\varphi(a)$, and so
$ \eta_{\varphi}$ is a representation of $E$.

Let $T\in L_{A}(E)$. We associate with $T$ a map on $E\otimes H$ by $x\otimes h\to Tx\otimes h$.
Since $\langle x\otimes h, Tx^{'}\otimes h^{'}\rangle=\langle T^{*}x\otimes h, x^{'}\otimes h^{'}\rangle$,
this map leaves invariant $N_{\varphi}$ so that it induces a map $\rho_{0}(T)$ on $(E\otimes H)/N_{\varphi}$. By
 \cite [Lemma 3.9]{ske1}, $\| \rho_{0}(T)\|=\|T\|$ and so $\rho_{0}(T)$ is bounded  and can be
extended to a bounded operator $\rho(T)$ on $K$.
Therefore $\rho: L_{A}(E)\to B(K)$ defined by $T\to \rho(T)$ is a representation of $L_{A}(E)$ on $K$.
\end{construction}

Now, we reformulate representations of Hilbert module
from the case of C*-algebras to the case of pro-C*-algebras.
Let $E$ and $F$ be two Hilbert modules over pro-C*-algebras $A$ and $B$,
respectively, and $\varphi:A \to B$ be a morphism of pro-C*-algebras.
A $\varphi$-morphism $\Phi:E\to B(H,K)$, where $\varphi:A \to B(H)$ is a
representation of $A$ is called a representation of $E$. If $p\in S(A)$ and $\varphi_{p}$ be
a representation of $A_{p}$ associated to $\varphi$, then it is
easy to see that the map $\Phi_{p}: E_{p}\to B(H,K)$, $\Phi_{p}(\sigma_{p}^{E}(x))=\Phi(x)$
is a $\varphi_{p}$-morphism. In this case, we say that $\Phi_{p}$
is a representation of $E_{p}$ associated to $\Phi$.
We can define non-degenerate representations and unitarily
equivalent representations for Hilbert modules over
pro-C*-algebras like Hilbert C*-modules case.

\begin{remark}\label{Cp3}
Suppose $A$ is a pro-C*-algebra, $E$ a Hilbert $A$-module
and $\varphi:A \to B(H)$ a representation of $A$ on some Hilbert space $H$.
Suppose $p \in S(A)$ and $\varphi_{p}$ is a representation of $A_{p}$ associated to $\varphi$.
By the above Construction $\varphi_{p}$ induces a representation
$\eta_{\varphi_{p}}:E_{p}\to B(H,K)$ of $E_{p}$ where $K$ is a
Hilbert space associated to $E_{p}\otimes H$.
It is easy to see that the map
$\eta_{\varphi}:E \to B(H,K)$, $\eta_{\varphi}(x)=\eta_{\varphi_{p}}(\sigma_{p}^{E}(x))$ is a
$\varphi$-morphism, i.e. it is a representation of $E$.
\end{remark}
The following theorem is a version of Stinespring representation theorem for pro-C*-algebras
that can be considered as a special case of KSGNS construction for completely positive maps on unital pro-C*-algebras,
by setting $B= \mathbb{C}$ in \cite[Theorem 4.6]{Joita2}).
We prove this theorem by using the concept of induced representations of Hilbert pro-C*-modules.
\begin{theorem} \label{Cp4}
Let $A$ be a unital pro-C*-algebras and $\varphi:A\to B(H)$ be a continuous operator
valued completely positive map. Then there exist a Hilbert space $H_{\varphi}$,
a unital representation $\pi_{\varphi}:A\to B(H_{\varphi})$ and a bounded
linear operator $V_{\varphi}\in B(H,H_{\varphi})$ such that
$\varphi(a)=V_{\varphi}^{*}\pi_{\varphi}(a)V_{\varphi}$ for all $a\in A$.
\end{theorem}

\begin{proof}
Suppose that $\pi^{'}_{\varphi}$ is Paschke's GNS construction associated to $\varphi$
and $X, \chi_{\varphi}$ and $\xi$ are as in Theorem \ref{Cp1}. Let $\iota$
be the identity map on $B(H)$. If we consider $\iota$ as a representation
of $B(H)$ on $H$ and apply Construction \ref{const1}, we get a Hilbert
space $H_{\varphi}$ (associated to $X\otimes H$), an induced representation
$\eta_{\iota}:X\to B(H,H_{\varphi})$ of $X$ and a representation
$\rho_{\varphi}:L_{B(H)}(X)\to B(H_{\varphi})$ of $L_{B(H)}(X)$.
We define $V_{\varphi}:=\eta_{\iota}(\xi)$  and $\pi_{\varphi}:=\rho_{\varphi} \circ \pi^{'}_{\varphi}$.
If $a\in A$ and $h\in H$, we have
\begin{eqnarray*}
V_{\varphi}^{*}\pi_{\varphi}(a)V_{\varphi}(h) = V_{\varphi}^{*}\pi_{\varphi}(a)(\xi\otimes h)
&=& V_{\varphi}^{*}(\pi^{'}_{\varphi}(a)\xi\otimes h)
\\ &=& \iota(\langle \xi,\pi^{'}_{\varphi}(a)\xi \rangle)h = \varphi(a)h.
\end{eqnarray*}
Hence, $\varphi(a)=V_{\varphi}^{*}\pi_{\varphi}(a)V_{\varphi}$ for all $a\in A$.
\end{proof}

In the rest of this section we establish \cite[Theorems 2.1 and 2.4 ]{Bhat} in the context of
pro-C*-algebra.
\begin{theorem}\label{Cp5}
Let $A$ be a unital pro-C*-algebra and $\varphi:A\to B(H)$ be a continuous
completely positive map. Let $E$ be a Hilbert $A$-module and $\Phi:E\to B(H,K)$ be a $\varphi$-morphism.
Then there exist triples $(\pi_{\varphi}, V_{\varphi}, H_{\varphi})$ and
$(\pi_{\Phi}, W_{\Phi}, K_{\Phi})$, where
\begin{enumerate}
\item $H_{\varphi}$ and $K_{\Phi}$ are Hilbert spaces;
\item $\pi_{\varphi}:A\to B(H_{\varphi})$ is a unital representation of $A$;
\item $\pi_{\Phi}:E\to B(H_{\varphi}, K_{\Phi})$ is a $\pi_{\varphi}$-morphism;
\item $V_{\varphi}:H\to H_{\varphi}$ and $W_{\Phi}:K\to K_{\Phi}$ are bounded linear operators such that
$\varphi(a)=V_{\varphi}^{*}\pi_{\varphi}(a)V_{\varphi},\  for\  all\  a\in A\  and\  \Phi(z)=W^{*}_{\Phi}\pi_{\Phi}(z)V_{\varphi},
\  for\  all\  z\in E.$
\end{enumerate}
\end{theorem}

\begin{proof}
Let $\pi^{'}_{\varphi}:A\to L_{B(H)}(X)$ be the Paschke's GNS construction associated
to $\varphi$. By continuity of $\pi^{'}_{\varphi}$, there exists $M>0$ and $p\in S(A)$
such that $\|\pi^{'}_{\varphi}(a)\|\leq Mp(a)$, for all $a\in A$.
Let $(\pi_{\varphi}, V_{\varphi}, H_{\varphi})$ be the Stinespring
triple for $\varphi$ as obtained in Theorem \ref{Cp4}.
Since $\pi_{\varphi}=\rho_{\varphi} \circ \pi^{'}_{\varphi}$,
we may consider $(\pi_{\varphi})_{p}$ as a representation of $A_{p}$
associated to $\pi_{\varphi}$. By Remark \ref{Cp3}, the
Stinespring representation $\pi_{\varphi}$ induces a representation
$\pi_{\Phi}:E\to B(H_{\varphi}, K_{\Phi})$ of $E$, where $K_{\Phi}$
is the Hilbert space associated to $E_{p}\otimes H_{\varphi}$.
Moreover, Lemma \ref{Cp2} implies the existence of an isometry
$v:E\otimes_{\pi^{'}_{\varphi}} X\to B(H,K)$ which is defined
by $v(x\otimes \xi)=\Phi(x)$ for all $x\in E$.
We consider the linear map $W_{0}:(E_{p}\otimes X)\otimes H\to K$ defined by
$W_{0}((\sigma^{E}_{p}(z)\otimes x)\otimes h)=v(z\otimes x)h,$
where $z\in E$, $x\in X$ and $h\in H$.
Let $z\in E$ and $\sigma^{E}_{p}(z)=0$. Since
$\|v(z\otimes x)\|^{2}=\|z\otimes x\|^{2}=\langle z\otimes x,z\otimes x\rangle=
\langle x, \pi^{'}_{\varphi}(\langle z,z\rangle)\rangle$,
we have $v(z\otimes x)=0$ which shows that $W_{0}$ is well-defined.
Moreover,
\begin{eqnarray*}
\|\sum_{i=1}^{n}(\sigma_{p}^{E}(z_{i})\otimes x_{i})\otimes h_{i}\|^{2} &=& \sum_{i,j=1}^{n}\langle  h_{i},\langle \sigma_{p}^{E}(z_{i})\otimes x_{i},\sigma_{p}^{E}(z_{j})\otimes x_{j}\rangle h_{j}\rangle
\\ &=& \sum_{i,j=1}^{n}\langle h_{i},\langle x_{i},\pi^{'}_{\varphi}(\langle z_{i},z_{j}\rangle)x_{j}\rangle h_{j}\rangle
\\ &=& \sum_{i,j=1}^{n}\langle h_{i},\langle z_{i}\otimes x_{i},z_{j}\otimes x_{j}\rangle h_{j}\rangle
\\ &=& \sum_{i,j=1}^{n}\langle h_{i},\langle v(z_{i}\otimes x_{i}),v(z_{j}\otimes x_{j})\rangle h_{j}\rangle
\\ &=& \sum_{i,j=1}^{n}\langle v(z_{i}\otimes x_{i})h_{i}, v(z_{j}\otimes x_{j})h_{j}\rangle
\\ &=& \|\sum_{i=1}^{n}v(z_{i}\otimes x_{i})h_{i}\|^{2}
\end{eqnarray*}
which implies that $W_{0}$ is an isometry.
Since 
$H_{\varphi}$ is the Hilbert space associated to $X\otimes H$, $W_{0}$ can be extended to a bounded
linear operator $W: K_{\Phi}\to K$. We define $W_{\Phi}:=W^{*}$, then
\begin{eqnarray*}
W^{*}_{\Phi}\pi_{\Phi}(z)V_{\varphi}(h) &=& W\pi_{\Phi}(z)(\xi \otimes h)
\\ &=& W(\pi_{\Phi})_{p} (\sigma^{E}_{p}(z))(\xi \otimes h)
\\ &=& W(\sigma^{E}_{p}(z)\otimes (\xi \otimes h))
\\ &=& W((\sigma^{E}_{p}(z)\otimes \xi) \otimes h)
\\ &=& v(z\otimes\xi)h = \Phi(z)h,
\end{eqnarray*}
for all $z\in E$ and $h\in H$. Hence, $\Phi(z)=W^{*}_{\Phi}\pi_{\Phi}(z)V_{\varphi}$ for all $z\in E$.
\end{proof}

\begin{remark}\label{Cp6}
Let $\varphi$ and $\Phi$ be as in Theorem \ref{Cp5} and  $q\in S(A)$.

(1) In the proof of Theorem \ref{Cp5}, if $(\pi_{\varphi})_{q}$ be a
representation of $A_{q}$ associated to $\pi_{\varphi}$ then we obtain a
representation $\tilde{\pi}_{\Phi}:E\to B(H_{\varphi},\tilde{K}_{\Phi})$,
where $\tilde{K}_{\Phi}$ is a Hilbert space associated to $E_{q}\otimes H_{\varphi}$.
It is easy to show that $\pi_{\Phi}$ and $\tilde{\pi}_{\Phi}$ are two
unitarily equivalent representations of $E$.

(2) The bounded linear operator $W_{\Phi}:K\to K_{\Phi}$ is a coisometry.
Indeed, for $z\in E,~x\in X$ and $h\in H$ we have
\begin{eqnarray*}
\langle W_{\Phi}^{*}(\sigma_{p}^{E}(z)\otimes x\otimes h),W_{\Phi}^{*}(\sigma_{p}^{E}(z)\otimes x\otimes h)\rangle &=&
\langle v(z\otimes x)h,v(z\otimes x)h\rangle
\\ &=& \langle v(z\otimes x)^{*}v(z\otimes x)h,h\rangle
\\ &=& \langle \langle v(z\otimes x),v(z\otimes x)\rangle h,h\rangle
\\ &=& \langle \langle z\otimes x,z\otimes x\rangle h,h\rangle
\\ &=& \langle h,\langle x,\pi_{\varphi}^{'}(\langle z,z\rangle)x\rangle h\rangle
\\ &=& \langle x\otimes h,\pi_{\varphi}^{'}(\langle z,z\rangle)x\otimes h\rangle
\\ &=&\langle x\otimes h,(\rho_{\varphi}\circ \pi_{\varphi}^{'})(\langle z,z\rangle)(x\otimes h)\rangle
\\ &=& \langle x\otimes h, \pi_{\varphi}(\langle z,z\rangle)(x\otimes h)\rangle
\\ &=& \langle x\otimes h,(\pi_{\varphi})_{p}(\langle \sigma_{p}^{E}(z),\sigma_{p}^{E}(z)\rangle)(x\otimes h)\rangle
\\ &=&\langle \sigma_{p}^{E}(z)\otimes x\otimes h,\sigma_{p}^{E}(z)\otimes x\otimes h\rangle
\end{eqnarray*}

(3) If $E$ is full then $\pi_{\Phi}:E\to B(H_{\varphi},K_{\Phi})$ is a
non-degenerate representation of $E$. To see this, let $z\in E$ and
$h_{\varphi}\in H_{\varphi}$ then $\pi_{\Phi}(z)(h_{\varphi})=\sigma_{p}^{E}(z)\otimes h_{\varphi}$.
Since $K_{\Phi}$ is a Hilbert space associated to $E_{p}\otimes H_{\varphi}$,
$[\pi_{\Phi}(E)(H_{\varphi})]=K_{\Phi}$. Moreover,
for $w\in E, x\in X$ and $h\in H$ we have
$$\pi_{\Phi}(z)^{*}(\sigma_{p}^{E}(w)\otimes x\otimes h)=
\pi_{\varphi}(\langle z,z\rangle)(x\otimes h)=\pi_{p}^{'}(\langle z,z\rangle)(x)\otimes h.$$
Since $E$ is full, $[\pi_{\varphi}^{'}(A)(X)]=X$. The Hilbert space $H_{\varphi}$ is
associated to $X\otimes H$ which follows that $\pi_{\Phi}(E)^{*}(K_{\Phi})=H_{\varphi}$.
\end{remark}

\begin{definition}
Let $\varphi$ and $\Phi$ be as in Theorem \ref{Cp5}. We say that the
pair $((\pi_{\varphi},V_{\varphi},H_{\varphi}),(\pi_{\Phi}, W_{\Phi},K_{\Phi}))$
is a {\it Stinespring representation} of
$(\varphi,\Phi)$ if conditions (1)-(3) of Theorem \ref{Cp5} are fulfilled.
Such a representation is said to be minimal if
\begin{enumerate}
\item $[\pi_{\varphi}(A)V_{\varphi}H]=H_{\varphi}$, and
\item $[\pi_{\Phi}(E)V_{\varphi}H]=K_{\Phi}$.
\end{enumerate}

\end{definition}

\begin{remark}
The pair $((\pi_{\varphi},V_{\varphi},H_{\varphi}),(\pi_{\Phi}, W_{\Phi},K_{\Phi}))$
obtained in Theorem \ref{Cp5} is a minimal representation for $(\varphi,\Phi)$ since
\begin{eqnarray*}
[\pi_{\varphi}(A)V_{\varphi}H] &=& [(\rho_{\varphi} \circ \pi^{'}_{\varphi})(A)(\xi\otimes H)]
\\ &=& [(\pi^{'}_{\varphi}(A)(\xi))\otimes H]
\\ &=& [\chi_{\varphi}\otimes H] = H_{\varphi}
\end{eqnarray*}
and
\begin{eqnarray*}
[\pi_{\Phi}(E)V_{\varphi}H]=[\pi_{\Phi}(E)\pi_{\varphi}(A)V_{\varphi}H] &=& [(\pi_{\Phi})_{p}(E_{p})H_{\varphi}]
\\ &=& [E_{p}\otimes H_{\varphi}] = K_{\Phi}.
\end{eqnarray*}
\end{remark}

The following result shows that the minimal Stinespring representation
is unique up to the unitarily equivalency.
\begin{proposition} \label{Cp7}
Let $\varphi$ and $\Phi$ be as in Theorem \ref{Cp5} and
$((\pi_{A},V^{'},H^{'}),(\pi_{E},W^{'},K^{'}))$ be a minimal
representation for $(\varphi,\Phi)$. Then there are two unitary
operators $U_{1}:H_{\varphi}\to H^{'}$ and $U_{2}:K_{\Phi}\to K^{'}$ such that
\begin{enumerate}
\item $V^{'}=U_{1}V_{\varphi}$,\ \ $U_{1}\pi_{\varphi}(a)=\pi_{A}(a)U_{1}$, for all $a\in A$ and
\item $W^{'}=U_{2}W_{\Phi}$,\ \ $U_{2}\pi_{\Phi}(z)=\pi_{E}(z)U_{1}$, for all $z\in E$.
\end{enumerate}
\end{proposition}

\begin{proof}
Existence $U_1$ and the statement (1) follow from \cite[Theorem 4.6 (2)]{Joita2}.
As in the proof of \cite[Theorem 2.4]{Bhat},
we define the linear map $U_{2}:{\rm span}(\pi_{\Phi}(E)V_{\varphi}H)\to {\rm span}(\pi_{E}(E)V^{'}H)$ by
$$U_{2}(\sum_{i=1}^{n}\pi_{\Phi}(z_{i})V_{\varphi}h_{i})=\sum_{i=1}^{n}\pi_{E}(z_{i})V^{'}h_{i},$$
for $z_{i}\in E,\ h_{i}\in H$ and $n\geq 1$. Then $U_{2}$ is a well-defined
isometry and so it can be extended to a unitary $U_{2}$ from $K_{\Phi}$ onto $K^{'}$ which
satisfies the statement (2).
\end{proof}

\section{Radon-Nikodym derivatives}
A Radon-Nikodym-type theorem for operator valued completely positive maps on Hilbert
C*-modules has been demonstrated in \cite{Joita5} by Joita. We are going to generalize
her definitions and results to the case of Hilbert modules over pro-C*-algebras.
Let $E$ be a full Hilbert module over a pro-C*-algebra $A$ and $H,K$ be two Hilbert spaces.
The set of all completely positive maps of $E$ into $B(H,K)$ will be denoted by $CP(E,B(H,K))$.
There is an equivalence relation on $CP(E,B(H,K))$ as follows.

\begin{definition}
Let $\Phi$ and $\Psi$ be in $CP(E,B(H,K))$. We say that $\Phi$ is equivalent to $\Psi$,
denoted by $\Phi\sim \Psi$, if $\Phi(x)^{*}\Phi(x)=\Psi(x)^{*}\Psi(x)$ for all $x\in E$.
\end{definition}

\begin{definition}
Let $\Phi$ and $\Psi$ be in $CP(E,B(H,K))$. We say that $\Psi$ is dominated by $\Phi$,
denoted by $\Psi\preceq\Phi$, if $\Psi(x)^{*}\Psi(x)\leq \Phi(x)^{*}\Phi(x)$ for all $x\in E$.
\end{definition}

\begin{remark}
The relation $` `\preceq"$ is reflexive and transitive and so is a preorder
relation on $CP(E,B(H,K))$. Moreover, if $\Phi,\Psi\in CP(E,B(H,K))$ then $\Phi\preceq\Psi$
and $\Psi\preceq\Phi$ if and only if $\Phi\sim\Psi$.
\end{remark}

In \cite{Arambasic}, Aramba\v{s}i\'c extended the definition of the commutant of a
C*-algebra to a Hilbert C*-module. We define a similar notion for Hilbert modules over pro-C*-algebras.

\begin{definition}
Let $A$ be a pro-C*-algebra and $\Phi:E\to B(H,K)$ be a representation of a Hilbert $A$- module $E$.
The commutant of $\Phi(E)$, which is denoted by $\Phi(E)^{'}$, is defined by
$$\{T\oplus S\in B(H\oplus K): T\in B(H), S\in B(K), \Phi(z)T=S\Phi(z), \Phi(z)^{*}S=T\Phi(z)^{*}, z\in E\}$$
in which, $(T\oplus S) (h\oplus k):=Th\oplus Sk$.
\end{definition}

If $T\oplus S\in \Phi(E)^{'}$, then $T\in \varphi(A)^{'}$, cf. \cite[Lemma 4.4]{Arambasic}.
If $\Phi$ is non-degenerate, then $S$ is uniquely determined by $T$, cf. \cite[Note 4.6]{Arambasic}.

\begin{lemma}
Let $\Phi\in CP(E,B(H,K))$ and $((\pi_{\varphi},V_{\varphi},H_{\varphi}),(\pi_{\Phi}, W_{\Phi},K_{\Phi}))$
be the Stinespring representation of $(\varphi,\Phi)$.
If $T\oplus S$ be a positive linear operator in $\pi_{\Phi}(E)^{'}$,
then the map $\Phi_{T\oplus S}:E\to B(H,K)$ defined by
$\Phi_{T\oplus S}(x)=W_{\Phi}^{*}\sqrt{T}\pi_{\Phi}(x)\sqrt{S}V_{\varphi}$ is completely positive.
\end{lemma}

\begin{proof}
As in proof of \cite[Lemma 2.10]{Joita5},
$\Phi_{T\oplus S}(x)^{*}\Phi_{T\oplus S}(y)=V_{\varphi}^{*}T^{2}\pi_{\varphi}(\langle x,y\rangle)V_{\varphi}$,
for all $x,y\in E$. Using \cite[Lemma 3.4.1]{Joita1} and the fact that $T^{2}\in \pi_{\varphi}(A)^{'}$,
we find $\Phi_{T\oplus S}(x)^{*}\Phi_{T\oplus S}(y)=\varphi_{T^{2}}(\langle x,y\rangle)$.
Indeed, the completely positive map associated to $\Phi_{T\oplus S}$ is $\varphi_{T^{2}}$.
\end{proof}

\begin{theorem}\label{Cp8}
Let $\Psi,\Phi\in CP(E,B(H,K))$. If $\Psi \preceq\Phi$, then there is a
unique positive linear operator $\Delta_{\Phi}(\Psi)$ in $\pi_{\Phi}(E)^{'}$
such that $\Psi \sim \Phi_{\sqrt{\Delta_{\Phi}(\Psi)}}$.
\end{theorem}

\begin{proof}
Let $((\pi_{\varphi},V_{\varphi},H_{\varphi}),(\pi_{\Phi}, W_{\Phi},K_{\Phi}))$
be the Stinespring representation of $(\varphi,\Phi)$. Continuity
of $\varphi$ and $\psi$ implies that there exist $p,q\in S(A)$ and $M,N>0$ such that
$\|\varphi(a)\|\leq Mp(a)$ and $\|\psi(a)\|\leq Nq(a),$ for all $a\in A$.
Let $r\in S(A)$ and $r\geq p,q$. The linear maps
$\varphi_{r}:A_{r}\to B(H), \varphi_{r}(\pi_{r}(a))=\varphi(a)$ and
$\psi_{r}:A_{r}\to B(H), \psi_{r}(\pi_{r}(a))=\psi(a)$ are
completely positive maps since,
$\sum_{i,j=1}^{n}\langle \varphi_{r}(\pi_{r}(a_{i})^{*}\pi_{r}(a_{j}))x_{j},x_{i}\rangle =
\sum_{i,j=1}^{n}\langle \varphi(a_{i}^{*}a_{j})x_{j},x_{i}\rangle\geq 0,
~{\rm for~all~}a_i \in A, ~  x_i\in H~ {\rm and} ~ 1 \leq i \leq n .$

The maps $\Phi_{r}: \sigma_{r}^{E}(x) \mapsto \Phi(x)$ and
$\Psi_{r}: \sigma_{r}^{E}(x) \mapsto \Psi(x)$
are in $CP(E_{r},B(H,K))$ and $\Psi_{r}\preceq\Phi_{r}$. Let
$((\pi_{\varphi_{r}},V_{\varphi_{r}},H_{\varphi_{r}}),
(\pi_{\Phi_{r}}, W_{\Phi_{r}},K_{\Phi_{r}}))$ be the Stinespring
representation of $(\varphi_{r},\Phi_{r})$. By the proof of \cite[Theorem 2.12]{Joita5},
there are unique positive linear operators
$\Delta_{1\Phi_{r}}(\Psi_{r}) \in B(H_{\varphi_{r}})$ and $\Delta_{2\Phi_{r}}(\Psi_{r}) \in B(K_{\Phi_{r}})$
such that $\Psi_{r} \sim \Phi_{{r}_{\sqrt{\Delta_{\Phi_{r}}(\Psi_{r})}}}$\,, where
$\Delta_{\Phi_{r}}(\Psi_{r})=\Delta_{1\Phi_{r}}(\Psi_{r})\oplus \Delta_{2\Phi_{r}}(\Psi_{r})\in \pi_{\Phi_{r}}(E)^{'}$
is the Radon-Nikodym derivative of $\Psi_{r}$ with respect to
$\Phi_{r}$. The pairs
$((\pi_{\varphi_{r}} \circ \pi_{r} ,V_{\varphi_{r}},H_{\varphi_{r}}),(\pi_{\Phi_{r}} \circ \sigma_{r}^{E},W_{\Phi_{r}},K_{\Phi_{r}}))$ and $((\pi_{\varphi},V_{\varphi},H_{\varphi}),(\pi_{\Phi}, W_{\Phi},K_{\Phi}))$
are two minimal Stinespring representations of $(\varphi,\Phi)$ and so, by
Proposition \ref{Cp7}, there are two unitary operators $U_{1}:H_{\varphi}\to H_{\varphi_{r}}$
and $U_{2}:K_{\Phi}\to K_{\Phi_{r}}$ such that $V_{\varphi_{r}}=U_{1}V_{\varphi}$\,,
$U_{1}\pi_{\varphi}(a)=(\pi_{\varphi_{r}}\circ \pi_{r})(a)U_{1}$ for all $a\in A$,
$W_{\Phi_{r}}=U_{2}W_{\Phi}$ and $U_{2}\pi_{\Phi}(z)=(\pi_{\Phi_{r}}\circ \sigma_{r})(z)U_{1}$
for all $z\in E$. Let $\Delta_{1\Phi}(\Psi)=U_{1}^{*}\Delta_{1\Phi_{r}}(\Psi_{r})U_{1}
$ and $\Delta_{2\Phi}(\Psi)=U_{2}^{*}\Delta_{2\Phi_{r}}(\Psi_{r})U_{2}$.
It is easy to see that $\Delta_{\Phi}(\Psi)=\Delta_{1\Phi}(\Psi)\oplus \Delta_{2\Phi}(\Psi)$
is a positive operator in $\pi_{\Phi}(E)^{'}$. For every $a\in A$, we have

\begin{eqnarray*}
\psi(a)=\psi_{r}(\pi_{r}(a)) &=& V_{\varphi_{r}}^{*}\Delta_{1\Phi_{r}}(\Psi_{r})\pi_{\varphi_{r}}(\pi_{r}(a))V_{\varphi_{r}}
\\ &=& V_{\varphi}^{*}U_{1}^{*}\Delta_{1\Phi_{r}}(\Psi_{r})U_{1}\pi_{\varphi}(a)U_{1}^{*}U_{1}V_{\varphi}
\\ &=& V_{\varphi}^{*}\Delta_{1\Phi}(\Psi)\pi_{\varphi}(a)V_{\varphi} = \varphi_{\Delta_{1\Phi}(\Psi)}(a).
\end{eqnarray*}
Indeed by the uniqueness of Radon-Nikodym derivative (\cite[Theorem 3.4.5]{Joita1}),
$\Delta_{1\Phi}(\Psi)$ is the Radon-Nikodym derivative of $\psi$ with respect to $\varphi$. Consequently,
$$\Phi_{\sqrt{\Delta_{\Phi}(\Psi)}}^{*}(x) \, \Phi_{\sqrt{\Delta_{\Phi}(\Psi)}}(x)=
\varphi_{\Delta_{1\Phi}(\Psi)}(\langle x,x\rangle)=\psi(\langle x,x\rangle)=\Psi(x)^{*}\Psi(x)$$
for every $x\in E$, which implies $\Psi\sim \Phi_{\sqrt{\Delta_{\Phi}(\Psi)}}$.
Let $T\oplus S$ be another positive linear map in $\pi_{\Phi}(E)^{'}$
such that $\Psi\sim \Phi_{\sqrt{T\oplus S}}$. Then
$\Phi_{\sqrt{\Delta_{\Phi}(\Psi)}}\sim \Phi_{\sqrt{T\oplus S}}$ and so
$\varphi_{\Delta_{1\Phi}(\Psi)}=\varphi_{T}$. By \cite[Theorem 3.4.5]{Joita1},
we deduce that $\Delta_{1\Phi}(\Psi)=T$. Since $\pi_{\Phi}$ is non-degenerate (Remark \ref{Cp6} (3)),
$\Delta_{2\Phi}(\Psi)$ and $S$ are uniquely determined by $\Delta_{1\Phi}(\Psi)$ and $T$, respectively.
Consequently, $\Delta_{2\Phi}(\Psi)=S$ and so $\Delta_{\Phi}(\Psi)=T \oplus S$.
\end{proof}

Suppose that $\Phi\in CP(E,B(H,K))$, $\hat{\Phi}=\{ \Psi\in CP(E,B(H,K)): \Phi\sim\Psi \}$
and $\Phi,\Psi\in CP(E,B(H,K))$. We write $\hat{\Psi}\leq \hat{\Phi}$, if
$\Psi\preceq\Phi$. We define
$$[0,\hat{\Phi}]:=\{ \hat{\Psi}: \Psi\in CP(E,B(H,K)), \Psi\preceq \Phi\}$$
and $$[0,I]_{\Phi} :=\{ T\oplus S \in \pi_{\Phi}(E)^{'}: 0\leq T\oplus S\leq I \}.$$

The following theorem can be thought as a Radon-Nikodym type theorem for
operator valued completely positive maps on Hilbert modules over pro-C*-algebras.

\begin{theorem}
Let $\Phi\in CP(E,B(H,K))$. The map $\hat{\Psi}\to \Delta_{\Phi}(\Psi)$
from $[0,\hat{\Phi}]$ to $[0,I]_{\Phi}$ is an order-preserving isomorphism.
\end{theorem}

\begin{proof}
The map is well-defined by Theorem \ref{Cp8}.  Let $\hat{\Psi}_1,\hat{\Psi}_2\in [0,\hat{\Phi}]$
and $\Delta_{\Phi}(\Psi_{1})=\Delta_{\Phi}(\Psi_{2})$.
Then $\Psi_{1}\sim\Phi_{\Delta_{\Phi}(\Psi_{1})}= \Phi_{\Delta_{\Phi}(\Psi_{2})}\sim\Psi_{2}$
and so it is injective. Let $T\oplus S\in [0,I]_{\Phi}$ then
$\Phi_{\sqrt{T\oplus S}}\in CP(E,B(H,K))$. Since $T\oplus S \in \pi_{\Phi}(E)^{'}$,
$T\in \pi_{\varphi}(A)^{'}$ and so by \cite[Theorem 3.4.5]{Joita1},
$\Phi_{\sqrt{T\oplus S}}(x)^{*}\Phi_{\sqrt{T\oplus S}}(x)=
\varphi_{T}(\langle x,x\rangle)\leq \varphi(\langle x,x\rangle)=\Phi(x)^{*}\Phi(x)$
for all $x\in E$. Thus $\Phi_{\sqrt{T\oplus S}}\preceq \Phi$.
Since $\Delta(\varphi_{T})=T$, $\Delta_{\Phi}(\Phi_{\sqrt{T\oplus S}})=T\oplus S$, i.e.,
the map is surjective.

If $\hat{\Psi}_1,\hat{\Psi}_2 \in [0,\hat{\Phi}]$ and $\hat{\Psi}_1 \leq \hat{\Psi}_2$,
then $\Psi_{1}\preceq \Psi_{2}$ and so $\psi_{1}\leq \psi_{2}$.
By \cite[Theorem 3.4.5]{Joita1}, we have $\Delta_{1\Phi}(\Psi_{1})\leq \Delta_{1\Phi}(\Psi_{2})$.
Since $\pi_{\Phi}$ is non-degenerate (Remark \ref{Cp6} (3)), $\Delta_{2\Phi}(\Psi_{1})$ and $\Delta_{2\Phi}(\Psi_{2})$ are uniquely determined by $\Delta_{1\Phi}(\Psi_{1})$ and $\Delta_{1\Phi}(\Psi_{2})$, respectively.
Consequently, $\Delta_{2\Phi}(\Psi_{1})\leq \Delta_{2\Phi}(\Psi_{2})$ and so $\Delta_{\Phi}(\Psi_{1})\leq \Delta_{\Phi}(\Psi_{2})$.
Conversely, let $T_{1}\oplus S_{1},T_{2}\oplus S_{2}\in [0,I]_{\Phi}$ and
$T_{1}\oplus S_{1}\leq T_{2}\oplus S_{2}$ then $T_{1},T_{2}\in [0,I]_{\varphi}$ and
$T_{1}\leq T_{2}$. By \cite[Theorem 3.4.5]{Joita1},
$\varphi_{{T_{1}}}\leq \varphi_{{T_{2}}}$ and so
$\Phi_{\sqrt{T_{1}\oplus S_{1}}}\preceq\Phi_{\sqrt{T_{2}\oplus S_{2}}}$.
\end{proof}

{\bf Acknowledgement}:  The second
author would like to thank Maria Joita who sent him
some copies of her recent books. The authors would like to thank
the referee for his/her careful reading and useful comments.

\end{document}